\documentclass[12pt]{amsart}
\usepackage{amsfonts}
\usepackage{amsfonts,latexsym,rawfonts,amsmath,amssymb,amsthm}
\usepackage[plainpages=false]{hyperref}

\usepackage{graphicx}

\RequirePackage{color}

\numberwithin{equation}{section}

\newcommand{\beq}{\begin{equation}}
\newcommand{\eeq}{\end{equation}}
\newcommand{\beqs}{\begin{eqnarray*}}
\newcommand{\eeqs}{\end{eqnarray*}}
\newcommand{\beqn}{\begin{eqnarray}}
\newcommand{\eeqn}{\end{eqnarray}}
\newcommand{\beqa}{\begin{array}}
\newcommand{\eeqa}{\end{array}}

\def\p{\partial }

\def\R{\Bbb R}

\def\Om{\Omega}
\def\pom{\p  \Omega}
\def\bom{\overline  \Omega}

\newtheorem{Proposition}{Proposition}[section]
\newtheorem{Theorem}[Proposition]{Theorem}
\newtheorem{Lemma}[Proposition]{Lemma}

\title  {On A Class of Degenerate And Singular Monge-Amp\`ere   Equations }\thanks{This work was supported by NSFC 11771237}

\begin{document}

\address{Huaiyu Jian: Department of Mathematics, Tsinghua University, Beijing 100084, China.}

\address{You Li:Department of Mathematics, Tsinghua University, Beijing 100084, China.}

\address{Xushan Tu:Department of Mathematics, Tsinghua University, Beijing 100084, China.}

\email{hjian@math.tsinghua.edu.cn; you-li16@mails.tsinghua.edu.cn; txs17@mails.tsinghua.edu.cn. }


\bibliographystyle{plain}

\maketitle

\baselineskip=15.8pt
\parskip=3pt

\centerline {\bf  Huaiyu Jian \ \ \ \ You Li \ \ \ \   Xushan Tu}
\centerline {Department of Mathematics, Tsinghua University}
\centerline {Beijing 100084, China}

\vskip20pt

\noindent {\bf Abstract}:
In this paper we shall prove  the existence, uniqueness  and global H$\ddot{o}$lder continuity for the Dirichlet problem of
a class of Monge-Amp\`ere type equations which may be degenerate and singular on the boundary of  convex domains.
 We  will establish  a relation of the H$\ddot{o}$lder exponent for the solutions with the convexity for the domains.
 \vskip20pt


\vskip20pt

\noindent {\bf  Running head}:
Degenerate And Singular  Monge-Amp\`ere  Equations
\vskip20pt

\baselineskip=15.8pt
\parskip=3pt

\newpage

\centerline {\bf On A Class of Degenerate And Singular Monge-Amp\`ere   Equations  }

 \vskip10pt

\centerline { Huaiyu Jian\ \ \ \ You Li\ \ \ \ Xushan Tu}

\maketitle

\baselineskip=15.8pt
\parskip=3.0pt

\section {Introduction}

In this paper we study the Monge-Amp\`ere type equation
\begin{equation}
\begin{split}\label{1.1}
\det  D^2 u&=  F(x, u)\ \  \text{in}\ \Om,\\
   u &=0\ \ \text{on}\ \pom,
\end{split}
 \end{equation}
where $\Om$ is a bounded convex domain in $R^n$ $(n\geq 2)$, and $F$
satisfies the following (1.2)-(1.3):
\begin{equation}\label{1.2}
\begin{split} F(x, t)\in C (\Om\times(-\infty,0)) \text{ is
non-decreasing in}\ t  \text{ for any}\ x\in \Om;
\end{split}
 \end{equation}
\begin{equation}\label{1.3}
\begin{split}
&\text{there are constants }\  A>0,\ \alpha\geq0,\ \beta\geq n+1 \ \text{ such that} \\
&0<F(x,  t )\leq A d_{x}^{\beta-n-1}|t|^{-\alpha} \ \ \ \forall (x,t) \in \Om\times(-\infty,0),\\
\end{split}
 \end{equation}
where $d_{x}=dist(x, \pom)$.  Obviously, this problem  is singular and degenerate at the boundary of the domain.

 The particular case of  problem (1.1)  includes a few geometric problems.  When $F= |t|^{-(n+2)}$  and $u$ is a solution to problem (1.1),
 then the Legendre transform of $u$  is a complete affine hyperbolic sphere \cite {[Ca2], [CY], [CY2], [JL], [JLW]}, and
  $(-u)^{-1}\sum u_{x_ix_j} dx_idx_j$ gives the Hilbert metric (Poincare metric) in the convex domain $\Om$ \cite {[LN]}. When $F =f(x)|t|^{-p},$
   problem (1.1) may be obtained from $L_p$-Minkowski problem \cite{[Lut]} and the Minkowski problem in centro-affine geometry\cite {[CW], [JLZ]}.  Also see p.440-441 in
  \cite{[JW]}.   Generally,  problem (1.1) can be applied to construct non-homogeneous complete Einstein-K$\ddot{a}$hler metrics on a tubular domain \cite {[CY], [CY1]}.

 Cheng and Yau in \cite{[CY]} proved that
 if $\Om$  is a  strictly convex $C^2$-domain  and  $F\in C^{k}$ ($k\geq3$)   satisfies (1.2)-(1.3),
then  problem (1.1) admits an unique  convex generalized solution  $u\in C(\bar \Om)$. Moreover,
 $u\in C^{k+1, \varepsilon}(\Om)\bigcap C^{\gamma}(\bar \Om)$  for  any   $\varepsilon \in ( 0 ,  1)$ and some
 $\gamma=C(\beta, \alpha, A, n, \partial\Om)\in (0, 1)$. We should emphasize that their methods need the strict convexity and the smoothness of  $\Om$,  and the differentiability of $F$.

In this paper we find that the global H$\ddot{o}$lder regularity  for   problem (1.1) is independent of the smoothness of $\Om$ and $F$, and  the H$\ddot{o}$lder exponent depends  only on the convexity of the domain.    As a result, we can remove the smoothness of  $\Om$ as well as  the differentiability of $F$ in \cite{[CY]}. Moreover,   using the concept of $(a, \eta)$ type introduced in \cite{[JL]}  to describe the convexity of the domain, we obtain a relation of the H$\ddot{o}$lder exponent for $u$  with the convexity for $\Omega$.

We have noticed that there are many papers on    global regularity for   equations of  Monge-Amp\`ere type. See, for example, \cite{[CNS], [F], [GT],[LS], [Sa], [TW], [U1]} and the references therein. But, generally speaking,  those results require that the   domain $\Omega$ should be  strictly  convex and $\partial \Omega  \in C^{1,1} $.

Our first result is stated as the following
\begin{Theorem}\label{1.1}
Supposed that $\Om$  is a  bounded convex domain in $R^n$ and  $F(x, t)$ satisfies (1.2)-(1.3).  Let
\begin{equation}\label{1.4}
\gamma_1:=\left\{\begin{array}{cc}
\frac{\beta-n+1}{n+\alpha}, & {\rm if }   \beta<\alpha+2n-1,\\
\text{any number in} (0, 1), & {\rm if } \beta\geq \alpha+2n-1.
\end{array}
\right.
  \end{equation}
 Then problem (1.1) admits an unique convex generalized solution  $u\in C^{\gamma_1}(\overline{\Om})$.
Furthermore, $u\in C^{2,\gamma_1}(\Om)$   if $F(x, t)\in C^{0,1}(\Om\times(-\infty,0))$.

\end{Theorem}

Here a generalized solution  means the well-known Alexandrov solution. See, for example,  \cite{[F], [G], [TW1]} for the details.

To  improve the regularity for the solution obtained in Theorem 1.1,  we  use the $(a, \eta)$ type in \cite{[JL]} to describe the convexity
of $\Om$. From now on, we denote
  $$x=(x_1, x_2, \dots , x_n)=(x', x_n), \ \
 x'=(x_{1},...,x_{n-1})$$ and
 $$  |x'|=\sqrt{x_{1}^{2}+...+x_{n-1}^{2}}.$$

 \noindent{\bf  Definition 1.1}. {\sl Supposed  that $\Om$ is a bounded convex domain in $R^n$, and $x_0\in \pom$.
 $x_{0}$ is called to be  $(a, \eta)$ type   if
there are numbers $a\in[1,+\infty)$ and $\eta>0$, after  translation and rotation transforms, we have
$$x_{0}=0  \ \   \text{and} \ \   \Om\subseteq\{x\in R^{n}|x_{n}\geq\eta|x'|^{a}\}.$$
$\Om$  is called   $(a, \eta)$ type domain if  every    point  of $\partial \Om$ is $(a, \eta)$ type.}

\noindent{\bf Remark 1.1}. The convexity requires that the number $a$ should be no less than 1. The  less is $a$, the more convex is the domain.
There is no $(a, \eta)$ type domain for $a\in [1,2)$, although part of
$\partial \Om$ may be $(a, \eta)$ type point for $a\in [1,2)$.

\noindent{\bf  Definition 1.2}. {\sl We say that a   domain
  $\Om$   in $R^n$ satisfies exterior (or interior) sphere condition with radius $R$ if   for  each $x_{0}\in \pom$,  there
 is a $B_R(y_0) \supseteq \Om $ (or $B_R(y_0) \subseteq \Om $, respectively) such that $\partial B_R(y_0) \bigcap  \pom\ni x_0$. }

In \cite{[JL]},   we have proved that $(2, \eta)$ type domain is equivalent to the domain satisfies exterior sphere condition.

 The following two theorems show the relation of the H$\ddot{o}$lder exponent for $u$  on $\bar \Om$ with the convexity for $\Omega$.

\begin{Theorem} \label {1.2} Supposed that $\Om$  is  (a, $\eta$) type domain  in $R^n$ with $a\in (2,+\infty)$, and  $F $  satisfies (1.2)-(1.3). Let Let
\begin{equation}\label{1.5}
\gamma_2:=\left\{\begin{array}{cc}
\frac{\beta-n+1}{n+\alpha}+\frac{2n-2}{a(n+\alpha)}, & {\rm if }   \beta<\alpha+2n-1-\frac{2n-2}{a},\\
\text{any number in} (0, 1), & {\rm if } \beta\geq \alpha+2n-1-\frac{2n-2}{a}.
\end{array}
\right.
  \end{equation}
 Then the convex generalized solution to problem (1.1)
\beq \label{1.6}   u\in C^{ \gamma_2 } (\overline{\Om}).\eeq
Furthermore $u\in C^{2, \gamma_2 }(\Om)$  if $F(x, t)\in C^{0,1}(\Om\times(-\infty,0))$.
\end{Theorem}
\vspace{0.5cm}

\begin{Theorem} \label {1.3} Let $\Om$  be a bounded convex domain in $R^n$ and $u$ be a convex generalized solution to problem (1.1).

(i) Suppose that $\Om$   satisfies exterior  sphere condition and $F$ satisfies (1.2)-(1.3).  Let
\begin{equation}\label{1.7}
\gamma_3:=\left\{\begin{array}{cc}
\frac{\beta}{n+\alpha}, & {\rm if }   \beta<\alpha+n,\\
\text{any number in} (0, 1), & {\rm if} \alpha+n\leq \beta< \alpha+n+1,\\
1, & {\rm if} \beta\geq \alpha+n+1.
\end{array}
\right.
\end{equation}
Then
\beq \label{1.8}
u\in C^{ \gamma_3 } (\overline{\Om}).\eeq
Furthermore $u\in C^{2, \gamma_3}(\Om)$  if $ F(x, t)\in C^{0,1}(\Om\times(-\infty,0)).$

(ii) If $\Om$  satisfies  interior sphere condition with radius $R$ and $F$ satisfies  (1.2) and
\beq \label{1.9}
A d_{x}^{\beta-n-1}|t|^{-\alpha} \leq F(x,  t ) , \ \forall (x,t)\in \Om\times(-\infty,0)
\eeq
 for some constants  $A>0$,
   then
\beq \label {1.10}   |u(y)|\geq C (d_y)^{\gamma_4}, \ \ \forall y\in \Om \eeq
for some constant $C=C(\beta, \alpha, A, n,  R)>0$, where
   \beq\label{1.11}    \gamma_4:=\frac{\beta}{n+\alpha}\in(0, 1),
\eeq
\end{Theorem}

\noindent{\bf Remark 1.2}.  The H$\ddot{o}$lder regularity result of Theorem 1.1 can be viewed as the limit case of Theorem 1.2 as $a\to \infty$. Theorem 1.3 (i) shows that Theorem 1.2 is true for $a=2$, since a $(2, \eta)$ type domain
is equivalent to that the domain satisfies exterior sphere condition.

In the following Sections 2, 3, and 4,  we will  prove Theorems 1.1, 1.2, and 1.3, respectively.

\section {Proof of Theorem 1.1}

We start at a primary result which is useful to proving that a convex function in $\Om$ is H$\ddot{o}$lder continuous in
$\bar \Om$.

\begin{Lemma}\label{2.1}
Let $\Om$ be a bounded convex domain and $u\in C(\overline{\Om})$ be a  convex  function in $\Om$ with $u|_{\pom}=0$. If there are $\gamma\in(0,1]$ and  $M>0$
such that
\beq \label {2.1} |u(x)|\leq M{d_{x}}^{\gamma},\ \ \forall x\in \Om ,\eeq
  then $u\in C^{\gamma}(\overline{\Om})$ and
$$|u|_{C^{\gamma}}(\overline{\Om})\leq  M\{1+[diam(\Om)]^{\gamma}\} .$$
 \end{Lemma}
\begin{proof} This was proved in \cite{[JL]}. Here we copy the arguments for the convenience.

For any two point $x_{1}$, $x_{2}\in \Om$, consider the line determined by $x_{1}$ and $x_{2}$.
The line will intersect $\pom$ at two points $y_{1}$ and $y_{2}$.
Without loss generality we assume the four points are   $y_{1}$, $x_{1}$, $x_{2}$, $y_{2}$  in order.
By restricted onto the line,  $u$ is one dimension convex function. By the monotonic proposition of convex functions, we have
$$|u(x_{2})-u(x_{1})|\leq \max \{|u(y_{1}+(x_{2}-x_{1}))-u(y_{1})|, \ |u(y_{2})-u(y_{2}-(x_{2}-x_{1}))|\}.$$
 Moreover, since $y_{1}\in\pom$, by the assumption (2.1) we have
\begin{equation*}
\begin{split}
|u(y_{1}+(x_{2}-x_{1}))-u(y_{1})|
=&|u(y_{1}+(x_{2}-x_{1}))|\\
\leq& M \{dist(y_{1}+x_{2}-x_{1}, \pom)\}^{\gamma}\\
\leq& M |x_{2}-x_{1}|^{\gamma}.\\
\end{split}
\end{equation*}
Similarly,
$$|u(y_{2})-u(y_{2}-(x_{2}-x_{1}))|\leq M |x_{2}-x_{1}|^{\gamma}.$$
The above three inequalities, together with (2.1), implies the desired result.
\end{proof}

To prove Theorem 1.1, we need an
  a priori estimate result as follows, which holds  without strictly convexity of $\Om$ or any smoothness  of $\Om$ and  of $F$.

\begin{Lemma} \label{2.2}
 Supposed that $\Om$  is a  bounded convex domain in $\R^n$ and  $F(x, t)$ satisfies (1.2) and (1.3). If
   $u$ is a   convex  generalized  solution to problem (1.1),
then   $u\in C^{\gamma_1}(\overline{\Om})$ and
\beq \label {2.2} |u|_{C^{\gamma_1}(\overline{\Om})}
\leq C(\alpha,\ \beta,\ A,\ diam(\Om),\ n),\eeq
\end{Lemma}
where $\gamma_1$ is given by (1.4).

\begin{proof} First, we may assume
\beq \label{2.3} \beta < \alpha+2n-1.\eeq
Since for the case $\beta\geq  \alpha+2n-1$, we take a $\hat \beta<\alpha+2n-1$ such that
$\frac{\hat \beta-n+1}{n+\alpha}$ can be any number in $(0, 1)$. (Note $n \geq 2$).  Obviously,
(1.3) still holds with $\beta$ replaced by $\hat \beta$. Hence, this case is reduced to the case (2.3).

 Next, we assume for the time being that
 \beq \label{2.4} 0\in \overline{\Om}\subseteq R_{+}^{n}.\eeq
 Then we are going to  construct a sub-solution to problem (1.1).

For brevity,  write $l=diam(\Om)$.   Set
$$W=-Mx_{n}^{\gamma}\cdot\sqrt{N^{2}l^{2}-r^{2}}$$
 where $r=\sqrt{x_{1}^{2}+...+x_{n-1}^{2}}$. We will choose positive constants    $\gamma$, $M$, $N$  such that  $W$ is an sub-solution  to problem (1.1)
 under the assumptions (2.3) and  (2.4).

For $i, j\in\{1, 2, ..., n-1\}$, write $W_{i}=\frac{\partial W}{\partial x_i}, W_{ij}=\frac{\partial^2 W}{\partial x_i\partial x_j}$. Then we have
 \begin{equation*}
\begin{split}
W_{i}&=Mx_{n}^{\gamma}\cdot \frac{x_{i}}{\sqrt{N^{2}l^{2}-r^{2}}},\\
W_{ij}&=Mx_{n}^{\gamma}\cdot \frac{1}{\sqrt{N^{2}l^{2}-r^{2}}}(\delta_{ij}+\frac{x_{i}x_{j}}{N^{2}l^{2}-r^{2}}),\\
W_{n}&=-M\gamma x_{n}^{\gamma-1}\cdot\sqrt{N^{2}l^{2}-r^{2}},\\
W_{in}&=M\gamma x_{n}^{\gamma-1}\cdot \frac{x_{i}}{\sqrt{N^{2}l^{2}-r^{2}}},\\
W_{nn}&=M\gamma(1-\gamma)x_{n}^{\gamma-2}\cdot\sqrt{N^{2}l^{2}-r^{2}}.
\end{split}
\end{equation*}
Denote
 $$D^{2}W:=\begin{pmatrix} G & \xi \\ \xi^{T} &  W_{nn} \end{pmatrix}$$
  where $\xi^{T}=(W_{n1}, ..., W_{n (n-1)})$, and $G$ is the $(n-1)$-order matrix. Then
$$det D^{2}W=det G\cdot(W_{nn}-\xi^{T}G^{-1}\xi).$$
  Since all the eigenvalues of $G$ are
$$Mx_{n}^{\gamma}\frac{1}{\sqrt{N^{2}l^{2}-r^{2}}},..., \ \ Mx_{n}^{\gamma}\frac{1}{\sqrt{N^{2}l^{2}-r^{2}}},\ \  Mx_{n}^{\gamma}\frac{N^{2}l^{2}}{(N^{2}l^{2}-r^{2})\sqrt{N^{2}l^{2}-r^{2}}},$$
$$det G=M^{n-1}N^{2}l^{2}x_{n}^{(n-1)\gamma}\cdot(\frac{1}{\sqrt{N^{2}l^{2}-r^{2}}})^{n+1}.$$
It is direct to verify that
$$ G\xi= \frac {N^{2}l^{2}Mx_{n}^{\gamma}}{(N^{2}l^{2}-r^{2})^{\frac{3}{2}}}\xi.$$
It follows that
\begin{equation*}
\begin{split}
\xi^{T}G^{-1}\xi &= \frac{(N^{2}l^{2}-r^{2})^{\frac{3}{2}}}{N^{2}l^{2}Mx_{n}^{\gamma}}|\xi|^2\\
 &=\frac{M\gamma^{2}}{N^{2}l^{2}} x_{n}^{\gamma-2}  r^{2} \sqrt{N^{2}l^{2}-r^{2}} .
\end{split}
\end{equation*}
Hence, we obtain that
  \begin{equation} \label {2.5}
\begin{split}
det D^{2}W &=det G (W_{nn}-\xi^{T}G^{-1}\xi) \\
 &=M^{n-1}N^{2}l^{2}x_{n}^{(n-1)\gamma} (\frac{1}{\sqrt{N^{2}l^{2}-r^{2}}})^{n+1}  M\gamma x_{n}^{\gamma-2}\sqrt{N^{2}l^{2}-r^{2}} \\
 & \ \ \cdot [1-(1+\frac{r^{2}}{N^{2}l^{2}})\gamma]\\
 &=M^{n}N^{2}l^{2}\gamma x_{n}^{n\gamma-2} (\frac{1}{\sqrt{N^{2}l^{2}-r^{2}}})^{n}  [1-(1+\frac{r^{2}}{N^{2}l^{2}})\gamma].
\end{split}
\end{equation}

We want to prove
\beq \label{2.6} detD^{2}W\geq F(x, W) \ \ \text{in }\ \ \Om.\eeq

   Since  (1.3) and (2.4) implies that
  $$F(x,  W)\leq A d_{x}^{\beta-n-1}|W|^{-\alpha}\leq A x_{n}^{\beta-n-1}|W|^{-\alpha},$$
 we see that (2.6) can be deduced from
   \beq \label{2.7}
   det D^{2} W\geq A x_{n}^{\beta-n-1}|W|^{-\alpha} \ \ \text{in }\ \ \Om ,\eeq
 which is equivalent to
 \beq \label {2.8}
  det D^{2} W\cdot \frac{1}{A} x_{n}^{n+1-\beta}|W|^{\alpha}\geq1 \ \ \text{in }\ \ \Om.\eeq
 By (2.5), (2.8) is nothing but
 \beq\label{2.9}
\frac{1}{A}M^{n+\alpha}N^{2}l^{2}\gamma x_{n}^{(n+\alpha)\gamma-(\beta- n+1)}[1-(1+\frac{r^{2}}{N^{2}l^{2}})\gamma]\cdot(\sqrt{N^{2}l^{2}-r^{2}}\ )^{\alpha-n}\geq 1
\ \ \text{in }\ \ \Om.\eeq

Now we  choose $\gamma=\frac{\beta- n+1}{n+\alpha}$ such that
$$(n+\alpha)\gamma-(\beta- n+1)=0.$$
 Since  $\gamma \in(0, 1)$ by (2.3) and
 $r=|x'|\leq diam(\Om)=l$ in $\Om$,  we first take $N=C(\gamma)$ large enough such that
 $$1-(1+\frac{r^{2}}{N^{2}l^{2}})\gamma>0.$$
Noting $N^{2}l^{2}-r^{2}\in [(N^{2}-1)l^{2},\  N^{2}l^{2}]$, \ we
then   take $M=C(A, \alpha, \gamma, N, n, l)$ large enough such that
$$\frac{1}{A}M^{n+\alpha}N^{2}l^{2}\gamma x_{n}^{(n+\alpha)\gamma-(\beta- n+1)}[1-(1+\frac{r^{2}}{N^{2}l^{2}})\gamma]\cdot(\sqrt{N^{2}l^{2}-r^{2}}\ )^{\alpha-n}\geq 1,$$
we obtain (2.9) and thus have proved (2.6).

Finally, for any point $y\in \Om$,  letting $z\in \pom$ be the nearest boundary point to $y$,
by some translations and rotations, we assume $z=0$, $\Om\subseteq R_{+}^{n}$ and the line $yz$  is the $x_{n}-axis$.
This is to say that (2.4) is satisfied. Therefore we have (2.6). Obviously,   $W\leq 0$ on $\bom$.
Hence, $W$ is a sub-solution to problem (1.1).
By comparison principle for generalized solutions (see \cite{[F],[G],[TW1]} for example), we have
 $$|u(y)|\leq|W(y)|\leq MNl y_{n}^{\frac{\beta- n+1}{n+\alpha}}=MNld_{y}^{\frac{\beta- n+1}{n+\alpha}},$$
which, together with Lemma 2.1,  implies the desired result (2.2).

Note that we have used the fact that problem (1.1) is invariant under translation and rotation transforms, since $det D^2u$ is invariant and $F(x,u)$ is transformed to the one
satisfying the same condition as $F$.  This fact will be again used  a few times in the following.
\end{proof}

\vskip 0.5cm

{\bf Proof of  Theorem 1.1.}\  We prove the theorem by three steps.

 {\bf Step 1.}\  Suppose that $\Om$ is bounded convex but  $F(x, t)\in C^{k}(\Om\times(-\infty,0))$ $(k\geq3)$ satisfies (1.2) and (1.3).

We choose a sequence of bounded and strictly convex domains $\{\Om_{i}\}$ such that
\beq \label {2.10} \Om_{i}\in C^{2}\ \ \text{and} \ \ \Om_{i}\subseteq\Om_{i+1}, i=1, 2, \cdots, \ \   \bigcup_{i=1}^{\infty}\Om_{i}=\Om. \eeq
Then by Theorem 5 in \cite{[CY]}, there exists a convex generalized solution $u_{i}$ to problem (1.1) in the domain $\Om_{i}$ for each $i$.
 We assume $u_{i}(x)=0$ for  all $x\in R^{n}\setminus\Om_{i}$.
By Lemma 2.2, We have the uniform estimations
\beq \label{2.11}|u_{i}|_{C^{\frac{\beta-n+1}{n+\alpha}}(\overline{\Om})}=|u_{i}|_{C^{\frac{\beta-n+1}{n+\alpha}}(\overline{\Om_{i}})}\leq C(\alpha,\ \beta,\ A,\ diam(\Om),\ n),
\eeq
which implies that there is a subsequence, still denoted by itself, convergent to a $u$ in the space $C(\overline{\Om})$. Moreover, by (2.11) again, we have
 $$|u|_{C^{\frac{\beta-n+1}{n+\alpha}}(\overline{\Om})}\leq C(\alpha,\ \beta,\ A,\ diam(\Om),\ n).$$
 By the well-known convergence result for  convex generalized solutions (see Lemma 1.6.1 in \cite{[G]} for example), we see that $u$ is a convex generalized solution to problem (1.1).

{\bf Step 2.}\  Drop the restriction on the smoothness for $F$.

Suppose $F_{j}\in C^{k}(\Om\times(-\infty,0))$  $(k\geq3)$ satisfy the same assumption as $F$ in the Step 1
and $F_{j}$ locally uniform convergence to $F$ in as $j\to \infty$. (For example we can take $F_{j}=F*\eta_{\varepsilon_{j}}$, $\varepsilon_{j}$ convergence to 0 as $j$ tend to $+\infty$.)
Then by the result of Step 1, for each $j$,  there exists a   convex generalized solution $u_{j}\in C^{\frac{\beta-n+1}{n+\alpha}}(\overline{\Om})$ to problem (1.1) with $F$  replaced by  $F_{j}$. Moreover, we
have
\beq \label {2.12} |u_{j}|_{C^{\frac{\beta-n+1}{n+\alpha}}(\overline{\Om})}\leq C(\alpha,\ \beta,\ A,\ diam(\Om),\ n)\eeq for all $j$. Using this estimate, Lemma 1.6.1 in \cite{[G]}, and the same argument as in Step 1,
we obtain a solution $u$ to problem (1.1), which is the limit of a subsequence of $u_j$ in the space space $C(\overline{\Om})$. Furthermore, we have
 $u\in C^{\frac{\beta-n+1}{n+\alpha}}(\overline{\Om})$ by (2.12).   The uniqueness for (1.1) is directly from the comparison principle (see \cite{[F],[G],[TW1]} for example).

{\bf Step 3.} \  We are going to prove $u\in C^{2,\ \frac{\beta-n+1}{n+\alpha}}(\Om)$  if  $F(x, t)\in C^{0,1}(\Om\times(-\infty,0))$.

It is enough to prove
 \beq \label{2.13} u\in C^{2,\ \frac{\beta-n+1}{n+\alpha}}(\overline{\Om_1})\eeq for any convex $\Om_1\subset\subset\Om$.

 Taking   a convex  $\Om'$ such that $\Om_1\subset\subset\Om'\subset\subset\Om$, if there  exists $z\in\overline{\Om'}\subset\Om$ such that $u(z)=0$,   then  $u\equiv0$ in $\Om$  by convexity and the boundary condition $u|_{\pom}=0$. Hence
 we obtain (2.13).  Otherwise,  $u(x)<0$ \ for all $x\in\overline{\Om'}$. Then
  $F(x, u(x))\in C^{\frac{\beta-n+1}{n+\alpha}}(\overline{\Om'})$ and is positive on  $\overline{\Om'}$. By  the Caffarelli's local $C^{2,\alpha}$ regularity in \cite{[Caf]} (also see \cite{[JW1]} for another proof),
 we obtain (2.13), too.

\vskip20pt

\section {Proof of Theorem 1.2}
 In this section we  establish the relation between the H$\ddot{o}$lder exponent and the convexity of the domain $\Om$ and thus prove Theorem 1.2.

 Assume that $\Om$ is a   $(a, \eta)$ type domain  with $a\in (2, \infty )$, $F$ satisfies (1.2)-(1.3),  and $u$ is the unique solution to problem (1.1) as in Theorem 1.1.
 To prove Theorem 1.2, it is sufficient to prove (1.6). See the Step 3 in the proof of Theorem 1.1.

As (2.3) we may assume
\beq \label {3.1} \beta <\alpha+2n-1-\frac{2n-2}{a}.\eeq
Hence, in the following we have
$$\gamma_2=\frac{\beta-n+1}{n+\alpha}+\frac{2n-2}{a(n+\alpha)}\in (0, 1).$$

 By Lemma 2.1,  (1.6) can be deduced from
\beq \label{3.2}
|u(y)|\leq C \ {d_{y}}^{ \gamma_2}, \ \ \forall y\in \Om
\eeq
for some  positive constant$C=C(a,n, \alpha, \eta, A, diam\Om)$.

 We are going to  prove (3.2).   For any $y\in \Om$, we can find $z\in \pom$, such that $|y-z|=d_{y}.$
 Since the domain $\Om$ is $(a,\eta)$ type and the problem (1.1) is invariant   under translation and rotation transforms,
 we may assume $z=0$, and take the line determined by $z$ and $y$
  as the $x_{n}-axis$ such that
  $$\Om\subseteq\{x\in R^{n}|x_{n}\geq\eta|x'|^{a}\}.$$
We will prove (3.2) by three steps.

{\bf Step 1.}  Let
$$W(x_{1}, ..., x_{n})=W(r, x_n)=-[(\frac{x_{n}}{\varepsilon})^{\frac{2}{a}}-x_{1}^{2}-...-x_{n-1}^{2}]^{\frac{1}{b}},$$
 where  $r=|x'|=\sqrt{x_1^2+\cdots, x_{n-1}^2}$, $b$ and $\varepsilon$ are positive constants to be determined.
 We want to find a   sufficient condition for which $W$ is a sub-solution to  problem (1.1).

For $i, j\in\{1, 2, ..., n-1\}$, by direct computation we have
\beq \label {3.3}
\begin{split}
W_{i}=&W_{r}\frac{x_{i}}{r},\\
W{ij}=&\frac{W_{r}}{r}\delta_{ij}+(W_{rr}-\frac{W_{r}}{r})\frac{x_{i}}{r}\frac{x_{j}}{r},\\
W_{in}=&W_{rn}\frac{x_{i}}{r}.\\
\end{split}
\eeq
Let $$D^{2}W:=\begin{pmatrix} G & \xi \\ \xi^{T} &  W_{nn} \end{pmatrix}$$
where $\xi^{T}=(W_{n1}, ..., W_{n (n-1)})$, and $G$ is the matrix of $n-1$ order all of which eigenvalues
are $$\frac{W_{r}}{r}, ..., \frac{W_{r}}{r}, W_{rr},$$ and one of which
  eigenvector  with respect to the eigenvalue $W_{rr}$
is $\xi$.  As obtaining (2.5), we have
  $$det D^{2}W=(\frac{W_{r}}{r})^{n-2}W_{rr}(W_{nn}-\frac{|W_{rn}|^{2}}{W_{rr}}).$$
Obviously, $W\leq 0$ on $\pom$.
 Therefore we conclude that {\sl $W$  is a sub-solution to  problem (1.1)
 if and only if
\beq \label{3.4} H[W]:=(\frac{W_{r}}{r})^{n-2}(W_{rr}W_{nn}-|W_{rn}|^{2})[F(x,W)]^{-1}\geq 1\ \ \text{in} \ \ \Om .\eeq}

We use the expression of $W$ to compute
\begin{equation*}
\begin{split}
W_{r}&=\frac{2}{b}((\frac{x_{n}}{\varepsilon})^{\frac{2}{a}}-r^{2})^{\frac{1}{b}-1}\cdot r,\\
W_{n}&=-\frac{2}{ab}((\frac{x_{n}}{\varepsilon})^{\frac{2}{a}}-r^{2})^{\frac{1}{b}-1}\cdot
(\frac{x_{n}}{\varepsilon})^{\frac{2}{a}-1}\cdot\frac{1}{\varepsilon},\\
W_{rr}&=\frac{4}{b}(1-\frac{1}{b})((\frac{x_{n}}{\varepsilon})^{\frac{2}{a}}-r^{2})^{\frac{1}{b}-2}\cdot r^{2}
+\frac{2}{b}((\frac{x_{n}}{\varepsilon})^{\frac{2}{a}}-r^{2})^{\frac{1}{b}-1},\\
W_{nn} &=\frac{4(b-1)}{a^{2}b^{2}}((\frac{x_{n}}{\varepsilon})^{\frac{2}{a}}-r^{2})^{\frac{1}{b}-2}\cdot
(\frac{x_{n}}{\varepsilon})^{\frac{4}{a}-2}\cdot(\frac{1}{\varepsilon})^{2} \\
& \ \ \ +\frac{2(a-2)}{a^{2}b}((\frac{x_{n}}{\varepsilon})^{\frac{2}{a}}-r^{2})^{\frac{1}{b}-1}\cdot
(\frac{x_{n}}{\varepsilon})^{\frac{2}{a}-2}\cdot(\frac{1}{\varepsilon})^{2},\\
W_{rn}&=\frac{4(1-b)}{ab^{2}}((\frac{x_{n}}{\varepsilon})^{\frac{2}{a}}-r^{2})^{\frac{1}{b}-2}\cdot
(\frac{x_{n}}{\varepsilon})^{\frac{2}{a}-1}\cdot r\cdot \frac{1}{\varepsilon}.
\end{split}
\end{equation*}
Using the expression of $W$  again we have
\beq \label{3.5} W_{r}=\frac{2}{b}|W|^{1-b}\cdot r,\eeq
\begin{equation*}
\begin{split}
W_{n}&=-\frac{2}{ab}|W|^{1-b}\cdot(\frac{x_{n}}{\varepsilon})^{\frac{2}{a}-1}\cdot\frac{1}{\varepsilon},\\
W_{rr}&=\frac{4(b-1)}{b^{2}}|W|^{1-2b}\cdot r^{2}+\frac{2}{b}|W|^{1-b},\\
W_{nn}&=\frac{4(b-1)}{a^{2}b^{2}}|W|^{1-2b}\cdot(\frac{x_{n}}{\varepsilon})^{\frac{4}{a}-2}\cdot\frac{1}{\varepsilon^{2}}
+\frac{2(a-2)}{a^{2}b}|W|^{1-b}\cdot(\frac{x_{n}}{\varepsilon})^{\frac{2}{a}-2}\cdot\frac{1}{\varepsilon^{2}},\\
W_{rn}&=\frac{4(1-b)}{ab^{2}}|W|^{1-2b}(\frac{x_{n}}{\varepsilon})^{\frac{2}{a}-1}\cdot r\cdot\frac{1}{\varepsilon}.
\end{split}
\end{equation*}
Hence,
\beq \label{3.6}
\begin{split}
W_{rr}\cdot W_{nn}-(W_{rn})^{2}&=\frac{8(a-2)(b-1)}{a^{2}b^{3}}|W|^{2-3b}\cdot(\frac{x_{n}}{\varepsilon})^{\frac{2}{a}-2}\cdot r^{2}\cdot(\frac{1}{\varepsilon})^{2}\\
&+\frac{8(b-1)}{a^{2}b^{3}}|W|^{2-3b}\cdot(\frac{x_{n}}{\varepsilon})^{\frac{4}{a}-2}\cdot(\frac{1}{\varepsilon})^{2}\\
&+\frac{4(a-2)}{a^{2}b^{2}}|W|^{2-2b}\cdot(\frac{x_{n}}{\varepsilon})^{\frac{2}{a}-2}\cdot(\frac{1}{\varepsilon})^{2}\\
& :=I_{1}+I_{2}+I_{3}.
\end{split}
\eeq

To estimate $I_{1}$, $I_{2}$ and $I_{3}$, \ we will choose a small $\delta=C(a, \alpha, \beta, n)>0$. Now for this $\delta$, we choose
a small  $ \varepsilon =C(\delta, a, \eta)>0$ such that
\beq \label {3.7}
\varepsilon (\frac{1}{\delta})^{\frac{a}{2}}\leq\eta .
\eeq
Then we have
\beq\label{3.8}
\Om\subseteq\{x\in R^{n}|x_{n}\geq\eta|x'|^{a}\}
\subseteq \{ x\in R^{n}| \delta(\frac{x_{n}}{\varepsilon})^{\frac{2}{a}}\geq r^{2}\}.
\eeq
 By (3.8) we have
  \beq \label{3.9}
  |W|^{b}=(\frac{x_{n}}{\varepsilon})^{\frac{2}{a}}-r^{2}\in[(1-\delta)(\frac{x_{n}}{\varepsilon})^{\frac{2}{a}}, \ (\frac{x_{n}}{\varepsilon})^{\frac{2}{a}}]\eeq

Since $a>2$, we have two case:  $a\geq \frac{2\alpha+2}{\beta-n+1}$ and $a< \frac{2\alpha+2}{\beta-n+1}$ if $\frac{2\alpha+2}{\beta-n+1}>2$.\\

{\bf Step 2.}  Assume that $\frac{2\alpha+2}{\beta-n+1}>2$ and $2<a< \frac{2\alpha+2}{\beta-n+1}$.  We want to  find  $b>1$ and $ \varepsilon>0$  such that
(3.4) is satisfied,  by which we will prove (3.2).

Since $a>2$ and $b>1$,  $I_{1}, I_{2}$ and $ I_{3}$ in (3.6) are all positive.
 $$W_{rr}\cdot W_{nn}-(W_{rn})^{2}\geq I_{2}=\frac{8(b-1)}{a^{2}b^{3}}|W|^{2-3b}\cdot(\frac{x_{n}}{\varepsilon})^{\frac{4}{a}-2}\cdot(\frac{1}{\varepsilon})^{2}.$$
 Observe that $d_x\leq x_n$ in $\Om$. Hence, by (1.3), (3.4) and (3.5) we obtain
 \begin{equation*}
 \begin{split}
 H[W]&=(\frac{W_{r}}{r})^{n-2}(W_{rr}W_{nn}-|W_{rn}|^{2}) [F(x,W)]^{-1}\\
&\geq (\frac{2}{b})^{n-2}\cdot|W|^{(1-b)(n-2)}\cdot\frac{8(b-1)}{a^{2}b^{3}}|W|^{2-3b}\cdot
(\frac{x_{n}}{\varepsilon})^{\frac{4}{a}-2}\cdot(\frac{1}{\varepsilon})^{2}\cdot
 \frac{1}{A}d_{x}^{n+1-\beta}|W|^{\alpha}\\
 &\geq(\frac{2}{b})^{n-2}\cdot|W|^{(1-b)(n-2)}\cdot\frac{8(b-1)}{a^{2}b^{3}}|W|^{2-3b}\cdot
(\frac{x_{n}}{\varepsilon})^{\frac{4}{a}-2}\cdot(\frac{1}{\varepsilon})^{2}\cdot\frac{1}{A}x_{n}^{n+1-\beta}|W|^{\alpha}.
\end{split}
\end{equation*}
it follows from (3.9) that

\begin{equation*}
 \begin{split}
&x_{n}\leq \varepsilon (\frac{1}{1-\delta})^{\frac{a}{2}}|W|^{\frac{ab}{2}} \ \ \  \text{in}\  \Om,\\
&(\frac{x_{n}}{\varepsilon})^{\frac{4}{a}-2}\geq [\ (\frac{1}{1-\delta})^{\frac{a}{2}}|W|^{\frac{ab}{2}}\ ]^{\frac{4}{a}-2},\\
&x_{n}^{n+1-\beta}\geq [\ \varepsilon (\frac{1}{1-\delta})^{\frac{a}{2}}|W|^{\frac{ab}{2}}\ ]^{n+1-\beta}
\end{split}
\end{equation*}

Therefore, we arrive at
\begin{equation*}
 \begin{split}
H[W]\geq&(\frac{2}{b})^{n-2}\cdot|W|^{(1-b)(n-2)}\cdot\frac{8(b-1)}{a^{2}b^{3}}|W|^{2-3b}\cdot
(\frac{1}{1-\delta})^{2-a}|W|^{\frac{ab}{2}(\frac{4}{a}-2)}\\
& \cdot(\frac{1}{\varepsilon})^{2}\cdot\frac{1}{A}\varepsilon^{n+1-\beta} (\frac{1}{1-\delta})^{\frac{a}{2}(n+1-\beta)}|W|^{\frac{ab}{2}(n+1-\beta)}|W|^{\alpha}\\
=&(\frac{1}{\varepsilon})^{\beta-n+1}\frac{1}{A}(\frac{2}{b})^{n-2}\cdot\frac{8(b-1)}{a^{2}b^{3}}\cdot
(\frac{1}{1-\delta})^{2-a+\frac{a}{2}(n+1-\beta)}\\
& \cdot |W|^{(1-b)(n-2)+2-3b+\frac{ab}{2}(\frac{4}{a}-2)+\frac{ab}{2}(n+1-\beta)+\alpha}.\\
 \end{split}
\end{equation*}

Now, we set
$$(1-b)(n-2)+2-3b+\frac{ab}{2}(\frac{4}{a}-2)+\frac{ab}{2}(n+1-\beta)+\alpha=0$$
which is equivalent to
  $$b=\frac{2(n+\alpha)}{a(\beta-n+1)+2n-2}.$$  Since  $a\in(2,\frac{2\alpha+2}{\beta-n+1}),$ we see that $b>1$ by (3.1).
  Observing that $\beta-n+1> 0$, we can
   choose $\varepsilon=C(a, \eta, A, \alpha, \beta, n)>0$ small enough again,
   such that $H[W]\geq1$. This  proves (3.4),  which is to say that
  $W$ is a sub-solution to problem (1.1). By
  comparison principle, we have
  $$|u(x)|\leq |W(x)|, \ \ \forall x\in \Om .$$
   Restricting this inequality onto the $x_{n}$ axis,
  we obtain $$|u(y)|\leq(\frac{y_n}{\varepsilon})^{\frac{2}{ab}}=({\frac{d_{y}}{\varepsilon}})^{\frac{\beta-n+1}{n+\alpha}+\frac{2n-2}{a(n+\alpha)}},$$
  which  is   (3.2) exactly.

\vspace{0.5cm}
{\bf Step 3. } Assume that   $a\geq \frac{2\alpha+2}{\beta-n+1}$. Note that $a>2$  by the assumption of the theorem.
 We will  find  $b\in (0,1)$ and $\varepsilon>0$  such that the function $W$
   is a sub-solution to problem (1.1), and thus prove (3.2).

By (3.9) we have
$$I_{1}\geq\frac{8(a-2)(b-1)}{a^{2}b^{3}}|W|^{2-3b}\cdot(\frac{x_{n}}{\varepsilon})^{\frac{2}{a}-2}\cdot\delta(\frac{x_{n}}
{\varepsilon})^{\frac{2}{a}}\cdot(\frac{1}{\varepsilon})^{2} =\delta(a-2)I_{2}.$$
Since  $a>2$, $b\in (0, 1)$ and  (3.9) yields
  $$(\frac{x_{n}}{\varepsilon})^{\frac{4}{a}-2}\leq|W|^{b(2-a)},$$
 we obtain
\begin{equation*}
\begin{split}
I_{1}+I_{2}&\geq(1+\delta(a-2))I_{2}\\
&\geq (1+\delta(a-2))\frac{8(b-1)}{a^{2}b^{3}}|W|^{2-3b}\cdot|W|^{b(2-a)}\cdot(\frac{1}{\varepsilon})^{2}\\
&=(1+\delta(a-2))\frac{8(b-1)}{a^{2}b^{3}}|W|^{2-b-ab}\cdot(\frac{1}{\varepsilon})^{2}.
\end{split}
\end{equation*}
Again by (3.9), we have
$$(\frac{x_{n}}{\varepsilon})^{\frac{2}{a}-2}\geq(\frac{1}{1-\delta})^{1-a}|W|^{b(1-a)}.$$
Hence, we have
\begin{equation*}
\begin{split}
I_{3}&\geq\frac{4(a-2)}{a^{2}b^{2}}|W|^{2-2b}\cdot(\frac{1}{1-\delta})^{1-a}|W|^{b(1-a)}\cdot(\frac{1}{\varepsilon})^{2}\\
&=\frac{4(a-2)}{a^{2}b^{2}}(\frac{1}{1-\delta})^{1-a}\cdot|W|^{2-b-ab}\cdot(\frac{1}{\varepsilon})^{2}.
\end{split}
\end{equation*}
Therefore, we obtain
\begin{equation*}
\begin{split}
W_{rr}\cdot W_{nn}-(W_{rn})^{2}
&=I_{1}+I_{2}+I_{3}\\
&\geq [(1+\delta(a-2))\frac{8(b-1)}{a^{2}b^{3}}+\frac{4(a-2)}{a^{2}b^{2}}(\frac{1}{1-\delta})^{1-a}]|W|^{2-b-ab}\cdot(\frac{1}{\varepsilon})^{2}\\
&:=\sigma(a, b, \delta)|W|^{2-b-ab}\cdot(\frac{1}{\varepsilon})^{2},
\end{split}
\end{equation*}
where $$\sigma(a, b, \delta)=(1+\delta(a-2))\frac{8(b-1)}{a^{2}b^{3}}+\frac{4(a-2)}{a^{2}b^{2}}(\frac{1}{1-\delta})^{1-a}.$$
Using above estimates, together with (1.3) and (3.9)  we have
\begin{equation*}
\begin{split}
H[W]&=(\frac{W_{r}}{r})^{n-2}(W_{rr}W_{nn}-|W_{rn}|^{2})(F(x, W))^{-1}\\
 &\geq(\frac{2}{b})^{n-2}\cdot|W|^{(1-b)(n-2)}\cdot \sigma(a, b, \delta)|W|^{2-b-ab}\cdot(\frac{1}{\varepsilon})^{2}\cdot (F(x, W))^{-1}\\
&\geq(\frac{2}{b})^{n-2}\cdot|W|^{(1-b)(n-2)}\cdot \sigma(a, b, \delta)|W|^{2-b-ab}\cdot(\frac{1}{\varepsilon})^{2}\cdot \frac{1}{A}d_{x}^{n+1-\beta}|W|^{\alpha}\\
&\geq(\frac{2}{b})^{n-2}\cdot|W|^{(1-b)(n-2)}\cdot \sigma(a, b, \delta)|W|^{2-b-ab}\cdot(\frac{1}{\varepsilon})^{2}\cdot \frac{1}{A}x_{n}^{n+1-\beta}|W|^{\alpha}\\
&=(\frac{1}{\varepsilon})^{\beta-n+1}(\frac{2}{b})^{n-2}\frac{1}{A}\cdot\sigma(a, b, \delta)|W|^{2-b-ab}\cdot|W|^{(1-b)(n-2)}\cdot  (\frac{x_{n}}{\varepsilon})^{n+1-\beta}|W|^{\alpha}\\
&\geq(\frac{1}{\varepsilon})^{\beta-n+1}(\frac{2}{b})^{n-2}\frac{1}{A}\cdot\sigma(a, b, \delta)|W|^{2-b-ab}\cdot|W|^{(1-b)(n-2)}\\
&\ \ \cdot (\frac{1}{1-\delta})^{\frac{a(n+1-\beta)}{2}}|W|^{\frac{ab(n+1-\beta)}{2}}\cdot |W|^{\alpha}\\
&=(\frac{1}{\varepsilon})^{\beta-n+1}(\frac{2}{b})^{n-2}\frac{1}{A}\cdot(\frac{1}{1-\delta})^{\frac{a(n+1-\beta)}{2}}\sigma(a, b, \delta)|W|^{2-b-ab+(1-b)(n-2)+\frac{ab(n+1-\beta)}{2}+\alpha} .
\end{split}
\end{equation*}

Now, we set
\beq \label {3.10} 2-b-ab+(1-b)(n-2)+\frac{ab(n+1-\beta)}{2}+\alpha=0,\eeq
which ie equivalent to
 $$ b=\frac{2(n+\alpha)}{a(\beta-n+1)+2n-2}.$$
   Since $a \geq \frac{2\alpha+2}{\beta-n+1} $,  we see that $b\in (0, 1]$.
Of course, we  also need
\beq \label {3.11} \sigma(a, b, \delta)=(1+\delta(a-2))\frac{8(b-1)}{a^{2}b^{3}}+\frac{4(a-2)}{a^{2}b^{2}}(\frac{1}{1-\delta})^{1-a}>0,\eeq
which is equivalent to
 \beq \label{3.12}(a-2)(1-\delta)^{a-1}>(1+\delta(a-2))(\frac{2(1-b)}{b}).\eeq
 Since $\gamma_2=\frac{\beta-n+1}{n+\alpha}+\frac{2n-2}{a(n+\alpha)}\in(0,1)$ by (3.1), we  see that
 $$a-2>\frac{a(\beta-n+1)+2n-2}{n+\alpha}-2 =(\frac{2(1-b)}{b}).$$  Using this and   taking   $\delta=C(a, \alpha, \beta, n)>0$ small enough,
 we obtain (3.12) and thus (3.11).

Finally,  choosing a positive $$\varepsilon=C(a, \eta, A, \alpha, \beta,  b(a, \alpha, \beta, n), \delta(a, \alpha, \beta, n))=C(a, \eta, A, \alpha, \beta,  n)$$
smaller if necessary, by (3.10) and (3.11) we obtain that  $H[W]\geq1$ in $\Om$,  which implies $W$ is an sub-solution to problem (1.1) by (3.4). As in the end of Step 2,  we have proved (3.2).

\vskip20pt

\section {Proof of Theorem 1.3}
As the proof of Theorem 1.2, the proof of (i) of Theorem 1.3 follows directly from
\beq \label{4.1}
|u(y)|\leq C \ {d_{y}}^{\gamma_3}, \ \ \forall y\in \Om
\eeq
for some  positive constant$C=C(a,n, \alpha, \eta, A, diam\Om)$.

   For any $y\in \Om$, we can find $z\in \pom$, such that $|y-z|=d_{y}.$   Since the domain $\Om$ satisfies exterior sphere condition with
    radius $R$ and the problem (1.1) is invariant   under translation and rotation transforms,
 we may assume
 \beq \label{4.2} z=\mathbf{0}\in \partial\Om\bigcap \partial B_R(y_0), \ \  \Om\subseteq B_R(y_0).\eeq

Since $z=\mathbf{0}$\ satisfies $|y-z|=d_{y}$, the tangent plane of $\Om$ at $z=\mathbf{0}$ is unique. And it is easy to check $y$ is on the line dertermined by $\mathbf{0}$ and $y_{0}$.
Hence $d_{y}=|y|=|y_{0}|-|y_{0}-y|=R-|y_{0}-y|$.\\

Consider the function
\beq \label{4.3} W(x)=-M (R^2-|x-y_0|^2)^b=-M (R^2-r^2)^b,\eeq
where $r=|x-y_0|$, $M$ and $b$ are positive constants to be determined later.  As (3.3), we obtain
that
$$det D^2W=(\frac{W_r}{r})^{n-1}W_{rr}.$$
But
$$W_r=2Mbr   (R^2-r^2)^{b-1},$$
$$W_{rr}=2Mb   (R^2-r^2)^{b-2}[R^2-(2b-1)r^2].$$
Hence
\beq \label{4.4}det D^2W=(2Mb)^n  (R^2-r^2)^{n(b-1)-1}[R^2-(2b-1)r^2].\eeq
Observing that $W\leq 0$ on $\partial \Om$, we see that $W$ is a sub-solution
to problem (1.1) if and only if
\beq \label{4.5} H[W]:=(2Mb)^n  (R^2-r^2)^{n(b-1)-1}[R^2-(2b-1)r^2][F(x,W)]^{-1}\geq 1\eeq
for all $x\in \Om$ and $r=|x-y_0|$.

First, we consider the case
\beq \label{4.6}\beta<n+\alpha+1.\eeq
 As (2.3), we need only to consider the case $\beta<n+\alpha.$
 We take \beq \label{4.7} b=\frac{\beta}{n+\alpha}=\gamma_3.\eeq Then in this case
  $b=\gamma_3\in (0, 1)$
 and  $|2b-1|<1$. Hence,
 \beq \label{4.8}R^2-(2b-1)r^2\geq(1-|2b-1|)R^{2}.\eeq
It follows from (4.2) that
\beq \label{4.9}d_x\leq R -|x-y_0|=R-r, \ \ \forall x\in \Om .\eeq
Therefore, by (1.3), (4.5), (4.8) and (4.9) that
\begin{equation}\label{4.10}
\begin{split}
H[W]& \geq (1-|2b-1|)R^{2}(2Mb)^n  (R^2-r^2)^{n(b-1)-1}\frac{1}{A}(d_x)^{n+1-\beta}|W|^{\alpha}\\
&\geq(1-|2b-1|)R^{2} \frac{1}{A}(2Mb)^n  (R^2-r^2)^{n(b-1)-1}(R-r)^{n+1-\beta}|W|^{\alpha}\\
&=(1-|2b-1|)R^{2}\frac{1}{A}M^{\alpha}(2Mb)^n  (R^2-r^2)^{n(b-1)+b\alpha-1}(R-r)^{n+1-\beta}\\
&=(1-|2b-1|)R^{2}\frac{1}{A}M^{n+\alpha}(2b)^n  (R+r)^{n(b-1)+b\alpha-1}(R-r)^{n(b-1)+b\alpha+n-\beta}.
\end{split}
\end{equation}
 Note that
\beq \label{4.11} n(b-1)+b\alpha+n-\beta=0\eeq
 by (4.7).
  Hence, by (4.10) and (4.11)
we can choose a large $M=C(A, b, R, \alpha, n, \beta)$ such that
\beq \label{4.12} H[W]\geq 1\ \ in \ \ \Om .\eeq

Next, we consider the case
 $$\beta\geq n+\alpha+1.$$
 In this case, we take
 $$ b=1=\gamma_3.$$
 Then , by (1.3) and (4.4) we have
\begin{equation*}
\begin{split}
H[W]& =(2M)^n  [F(x,W)]^{-1}\\
& \geq \frac{1}{A}2^n M^{n+\alpha}  (R+r)^{ \alpha}(R-r)^{\alpha+n+1-\beta} \\
& = \frac{1}{A} 2^n M^{n+\alpha} (R+r)^{ \alpha}.
\end{split}
\end{equation*}
Therefore, (4.12) still holds true.

 To sum up, we have obtained (4.5).  By comparison principle, we see that
\beq \label{4.13} W(x)\leq u(x)\leq 0.\eeq
In particular, we obtain that
$$|u(y)|\leq |W(y)|=M(R+|y-y_0|)^{\gamma_3} (R-|y-y_0|)^{\gamma_3} \leq M(2R)^{\gamma_3}(d_y)^{\gamma_3}.$$
This is desired (4.1) and hence we have proved the (i) of Theorem 1.3.
\vskip 0.5cm

To prove (ii) of Theorem 1.3, we notice that $u\in C(\bar \Om)$ and $u<0$ in $\Om$
and $u=0$ on $\pom$. By comparing the graph of the convex function $ u$ with the  cone whose   vortex is $(x_0, u(x_0))$ and  whose upper bottom is $\bar \Om$, where $u(x_0)=\min_{\bar\Om}u$, we
see easily that (1.10) is true for $\gamma_4\geq 1$. Hence, we need only to consider that case $\gamma_4< 1$ in the following, which implies that $\beta< n+1$.

 Since  (1.10)
holds naturally for all $y\in \{x\in \Om: d_x\geq \frac{R}{2}\}$, where $R$ is the radius of the interior sphere for
the $\Om$. Hence, it is sufficient to prove
\beq \label{4.14}  |u(y)|\geq (d_y)^{\gamma_4}, \ \  \forall y\in \{x\in \Om: d_x< \frac{R}{2}\}.\eeq

Take such a $y$.  We can find $z\in \pom$, such that $|y-z|=d_{y}.$
 we may assume
 \beq \label{4.15} z=\mathbf{0}\in \partial\Om\bigcap \partial B_R(y_0), \ \   B_R(y_0)\subseteq\Om.\eeq

Since the tangent plane of $\Om$ at $z=\mathbf{0}$ is unique. And it is easy to check $y$ is on the line determined by $\mathbf{0}$ and $y_{0}$.
Hence $d_{y}=|y|=|y_{0}|-|y_{0}-y|=R-|y_{0}-y|$.\\

Observing that in this case, instead of (4.8) we have
\beq \label{4.16}d_x\geq R -|x-y_0|=R-r, \ \ \forall x\in  B_R(y_0).\eeq
First, we require $b\in(0, 1)$, which implies $2b-1\in(-1, 1)$.  Similarly to the arguments of (i),\  by (4.16)  we find that the function $W$, given by (4.3),
satisfies
\begin{equation}\label{4.17}
\begin{split}
H[W]& \leq \frac{1}{A}(2Mb)^n [R^2- r^2)]^{n(b-1)-1} [R^2-(2b-1)r^2)]d_{x}^{n+1-\beta}|W|^{\alpha}\\
&\leq \frac{1}{A}M^{\alpha}(2Mb)^n2R^{2} [R^2- r^2)]^{n(b-1)-1+b\alpha} (R-r)^{n+1-\beta}\\
&\leq \frac{1}{A}M^{\alpha+n}(2b)^n2R^{2}(2R)^{n(b-1)-1+b\alpha} (R- r)^{n(b-1)+b\alpha+n-\beta}.
\end{split}
\end{equation}
Taking  $b=\frac{\beta}{n+\alpha}=\gamma_{4}\in (0, 1)$ we have
\beq \label{4.18} n(b-1)+b\alpha+n-\beta = 0.\eeq
Using (4.17)-(4.18),  we see that $W$ is a super-solution to problem (1.1) in the domain
$B_R(y_0)$ for sufficiently  small $M=C(A, b, R, \alpha, n, \beta)>0$. Since $u$ is a solution on $\Om$ and $u|_{\partial B_R(y_0)}\leq0$,  thus $u$ is a sub-solution on $B_R(y_0)$.  Therefore, we have
\begin{equation*}
\begin{split}
|u(y)|&\geq |W(y)|\\
      &=M(R+|y-y_0|)^{\gamma_4}(R-|y-y_0|)^{\gamma_4}\\
      &\geq MR^{\gamma_3}(d_y)^{\gamma_4},
\end{split}
\end{equation*}
Which is the desired (4.14) exactly. In this way, the proof of Theorem 1.3 has been completed.

\end{document}